\documentclass[11pt]{article}

\usepackage[utf8]{inputenc}
\usepackage[T1]{fontenc}
\usepackage[top=2.5cm,bottom=2.5cm,left=2.5cm,right=2.5cm]{geometry}
\usepackage{amsmath,amssymb,amsthm}
\usepackage{tikz}
\usepackage[colorlinks=true,linkcolor=red,citecolor=blue,urlcolor=magenta]{hyperref}
\allowdisplaybreaks[4]

\newtheorem{theorem}{Theorem}[section]
\newtheorem{lemma}[theorem]{Lemma}
\newtheorem{problem}[theorem]{Problem}
\newtheorem{definition}[theorem]{Definition}
\newtheorem{claim}[theorem]{Claim}

\newcommand{\tw}{\operatorname{tw}}
\newcommand{\argmax}{\operatorname*{arg\,max}}
\newcommand{\red}{\mathrm{red}}
\newcommand{\blue}{\mathrm{blue}}

\begin{document}

\title{Two-color Ramsey lower bounds for bounded degree hypergraphs}
\author{
Chunchao Fan\footnote{School of Mathematical Sciences, University of Science and Technology of China, Hefei, 230026, P.~R.~China. Email: {\tt 1807951575@qq.com}. This work was initiated while the author was at the Center for Discrete Mathematics, Fuzhou University, Fuzhou, 350108, P.~R.~China.}
\qquad
Qizhong Lin\footnote{Center for Discrete Mathematics, Fuzhou University, Fuzhou, 350108, P.~R.~China. Email: {\tt linqizhong@fzu.edu.cn}.}
}
\date{}
\maketitle

\begin{abstract}
We consider Ramsey numbers of bounded-degree uniform hypergraphs. In particular, we prove that for every $k\ge3$, there exists a constant $c_k>0$ such that, for all sufficiently large $\Delta$ and every $n\ge2^\Delta$, there is a $k$-uniform $n$-vertex hypergraph $H$ with maximum degree at most $\Delta$ satisfying
\[
        r(H)\ge \tw_{k-1}\!\bigl(c_k\Delta\log\log\Delta\bigr)\,n.
\]
Here $\tw_j$ denotes the tower function of height $j$. This constitutes the first progress towards a problem posed by Conlon, Fox and Sudakov.

\end{abstract}

\section{Introduction}\label{sec:introduction}

Ramsey theory, originating from Ramsey's seminal work~\cite{Ramsey}, studies the emergence of structure in large colored sets. A central object of study is the Ramsey number $r(H)$, defined as the smallest integer $N$ such that every red/blue coloring of the edges of the complete $k$-uniform hypergraph $K_N^{(k)}$ contains a monochromatic copy of $H$. Determining the growth of $r(H)$ for various $H$ has been a driving force in combinatorics.

For complete graphs, the classical bounds $2^{n/2}<r(K_n)<4^n$ were established by Erd\H{o}s~\cite{Erdos1947} and by Erd\H{o}s and Szekeres~\cite{ES}, with recent breakthroughs~\cite{CGMS,Gupta} improving the upper bound. For $k$-uniform hypergraphs with $k\ge3$, the situation is drastically different: the Ramsey number $r(K_n^{(k)})$ exhibits tower-type growth,
\[
        \tw_{k-1}(\Omega_k(n^2))
        \le r(K_n^{(k)})
        \le \tw_k(O(n)),
\]
where the lower bound follows from the stepping-up lemma of Erd\H{o}s and Hajnal (see, e.g.,~\cite{GRS}) and the upper bound from Erd\H{o}s and Rado~\cite{ERado}.

A major line of research concerns Ramsey numbers of graphs and hypergraphs with bounded maximum degree. For graphs, a foundational result of Chv\'atal, R\"odl, Szemer\'edi and Trotter~\cite{CRST} shows that bounded-degree graphs have linear Ramsey numbers. The dependence on the maximum degree was subsequently refined, and the best general upper bound currently known is $2^{O(\Delta\log\Delta)}n$ by Conlon, Fox and Sudakov~\cite{CFS12}.

For hypergraphs, the linearity phenomenon persists. Cooley, Fountoulakis, K\"uhn and Osthus~\cite{CFKO1,CFKO2} and, independently, Nagle, Olsen, R\"odl and Schacht~\cite{NORS} proved that, for every fixed $k$ and $\Delta$, there exists a constant $C=C(k,\Delta)>0$ such that every $n$-vertex $k$-graph $H$ of maximum degree at most $\Delta$ satisfies
\[
        r(H)\le Cn.
\]
More recently, Conlon, Fox and Sudakov~\cite{CFS09} proved, using dependent random choice, that for every fixed $k\ge4$ there exists a constant $c=c(k)>0$ such that
\[
        r(H)\le \tw_k(c\Delta)n,
\]
and that, in the $3$-uniform case, there exists an absolute constant $c>0$ such that
\[
        r(H)\le \tw_3(c\Delta\log\Delta)n.
\]

In the opposite direction, Conlon, Fox and Sudakov~\cite{CFS09} posed the following natural problem.

\begin{problem}\label{problem:CFS}
Is it true that for every $k\ge3$ and $\Delta$, and for all sufficiently large $n$, there exists a $k$-uniform hypergraph $H$ with maximum degree at most $\Delta$ and $n$ vertices such that
\[
        r(H)\ge \tw_k(c\Delta)n,
\]
where $c>0$ depends only on $k$?
\end{problem}

The key new ingredient first appears at the $3$-uniform base level.

\begin{theorem}\label{thm:base-three}
There exists an absolute constant $c>0$ such that, for all sufficiently large integers $\Delta$ and every $n\ge2^\Delta$, there exists a $3$-uniform $n$-vertex hypergraph $H$ with maximum degree at most $\Delta$ such that
\[
        r(H)\ge 2^{c\Delta\log\log\Delta}\,n.
\]
\end{theorem}

Theorem~\ref{thm:base-three} provides the starting point for the higher-uniformity result. Applying the stepping-up argument recursively, each increase in the uniformity raises the tower height by one, leading to the following general theorem.

\begin{theorem}\label{main}
For any $k\ge3$, there exists a constant $c_k>0$ such that for any integers $\Delta\ge1/c_k$ and $n\ge2^\Delta$, there exists a $k$-uniform $n$-vertex hypergraph $H$ with maximum degree at most $\Delta$ such that
\[
        r(H)\ge \tw_{k-1}\!\bigl(c_k\Delta\log\log\Delta\bigr)\,n.
\]
\end{theorem}

Our proof adopts a stepping-up strategy in the spirit of the framework introduced by Brada\v{c}, Hunter and Sudakov~\cite{BHS}. The principal novelty here is a robust 3-uniform base, which we construct by combining an ordered palette with a bounded-degree random matching; a separate small complete hypergraph is used to eliminate the blue case. This base case is then propagated through the stepping-up recursion, yielding the additional $\log\log\Delta$ factor.

The paper is organized as follows. Section~\ref{sec:stepping} introduces the stepping-up coloring and the relevant embedding notions. Section~\ref{sec:technical} establishes the ordered-palette estimates, the blue obstruction, and the robust random construction. Section~\ref{sec:proofs} first proves the $3$-uniform base theorem and then proves Theorem~\ref{main} by combining these ingredients with the reverse-induction argument.

\section{Stepping up and colorings}\label{sec:stepping}

We use $[n]$ to denote $\{1,2,\ldots,n\}$. For integers $a<b$, write $[a,b]=\{a,a+1,\ldots,b\}$ and $[a,b)=\{a,a+1,\ldots,b-1\}$. Throughout the paper, we omit floor and ceiling signs whenever they have no effect on the argument. Define the tower function recursively by $\tw_1(x)=x$ and $\tw_{i+1}(x)=2^{\tw_i(x)}$.

Fix $k\ge3$ throughout Sections~\ref{sec:stepping} and~\ref{sec:technical}, and put $\eta=1/(4k)$. Let $c_0=c_0(k)>0$ be a sufficiently small constant, chosen after Lemma~\ref{lem:robust-random}. For a sufficiently large positive integer $m$, set
\[
        q=q(m)=\lfloor m^\eta\rfloor,\qquad
        \rho=\rho(m)=(\log m)^{-1/4},\qquad
        M_3(m)=\lfloor\exp(c_0m\log\log m)\rfloor.
\]
For $\ell\ge4$, define recursively
\begin{equation}\label{eq:M-recursion}
        M_\ell(m)=2^{M_{\ell-1}(m)-1}.
\end{equation}

We shall later fix a pair-coloring $\xi_m:\binom{[0,M_3(m))}{2}\to[q]$ provided by Lemma~\ref{lem:ordered-palette}. Throughout this section, whenever $x_1\le\cdots\le x_\ell$, the notation $\{x_1,\ldots,x_\ell\}$ denotes the $\ell$-element multiset represented by this nondecreasing list; in particular, repetitions are retained. The colorings $\phi_m^{(\ell)}$ below are defined on such multisets. For $0\le x_1\le x_2\le x_3<M_3(m)$, define
\[
\phi_m^{(3)}(\{x_1,x_2,x_3\})=
\begin{cases}
\red,&\text{if }x_1,x_2,x_3\text{ are not all distinct},\\
\red,&\text{if }x_1<x_2<x_3\text{ and }\xi_m(x_1x_2)=\xi_m(x_1x_3),\\
\blue,&\text{otherwise}.
\end{cases}
\]

We now recall the stepping-up notation. For a nonnegative integer $x$, write
\[
        x=\sum_{j=0}^\infty a_j2^j,
        \qquad a_j\in\{0,1\},
\]
for its binary representation. For $i\ge1$, set
$\operatorname{bit}(x,i):=a_{i-1}$. For distinct $x,y\ge0$, define
\[
        \delta(x,y)
        :=\max\{i\in\mathbb Z_{>0}:
        \operatorname{bit}(x,i)\ne\operatorname{bit}(y,i)\},
\]
and, by convention, put $\delta(x,x):=0$. Thus $\delta(x,y)>0$ whenever
$x\ne y$. For $x_1\le\cdots\le x_t$, write
$\delta(\{x_1,\ldots,x_t\})=(\delta_1,\ldots,\delta_{t-1})$,
where $\delta_i=\delta(x_i,x_{i+1})$. We use the following standard properties.

\medskip
\noindent\textbf{Property A.} For distinct $x,y$, we have $x<y$ if and only if $\operatorname{bit}(x,\delta(x,y))<\operatorname{bit}(y,\delta(x,y))$.

\medskip
\noindent\textbf{Property B.} If $x\le y\le z$ and $x<z$, then $\delta(x,y)\ne\delta(y,z)$.

\medskip
\noindent\textbf{Property C.} For $x_1\le\cdots\le x_t$, we have
\[
        \delta(x_1,x_t)=\max_{1\le i\le t-1}\delta_i.
\]
Moreover, if $x_1<x_t$, then this maximum is attained at a unique index. Indeed, deleting consecutive repetitions gives a strictly increasing sequence to which the usual stepping-up property applies; reinserting the repetitions only adds zero $\delta$-values, whereas the maximum is positive.

For $\ell\ge4$, let $0\le x_1\le\cdots\le x_\ell<M_\ell(m)$ and write $\delta(\{x_1,\ldots,x_\ell\})=(\delta_1,\ldots,\delta_{\ell-1})$. A sequence is called monotone if it is nondecreasing or nonincreasing. If this $\delta$-sequence is not monotone, then $x_1<x_\ell$, so Property~C implies that its maximum is unique. Whenever the maximum is unique, we write $\argmax_i\delta_i$ for the unique index at which it is attained. Thus the following definition is unambiguous:
\[
\phi_m^{(\ell)}(\{x_1,\ldots,x_\ell\})=
\begin{cases}
\phi_m^{(\ell-1)}(\{\delta_1,\ldots,\delta_{\ell-1}\}),
    &\text{if }(\delta_1,\ldots,\delta_{\ell-1})\text{ is monotone},\\
\red,&\text{if it is not monotone and }\argmax_i\delta_i\in\{1,\ell-1\},\\
\blue,&\text{if }\argmax_i\delta_i\in[2,\ell-2].
\end{cases}
\]

Given a positive real $b$, an $\ell$-graph $G$, and a map $h:V(G)\to[0,M_\ell(m))$, we say that $h$ is an \emph{embedding} of $G$ into $\phi_m^{(\ell)}[b]$ if $|h^{-1}(y)|\le b$ for every $y\in[0,M_\ell(m))$. It is an \emph{almost monochromatic embedding} in color $c\in\{\red,\blue\}$ if, for every edge $e=\{v_1,\ldots,v_\ell\}\in E(G)$, either the images $h(v_1),\ldots,h(v_\ell)$ are not all distinct or $\phi_m^{(\ell)}(\{h(v_1),\ldots,h(v_\ell)\})=c$. It is a \emph{monochromatic embedding} in color $c$ if this equality holds for every edge, including those whose images are not all distinct.

If there is no monochromatic embedding of an $\ell$-graph $G$ into $\phi_m^{(\ell)}[b]$, then
\[
        r(G)>M_\ell(m)\lfloor b\rfloor.
\]
Indeed, color the complete $\ell$-graph on $[0,M_\ell(m))\times[\lfloor b\rfloor]$ according to the first coordinates. Every vertex of $G$ mapped into a fixed first-coordinate fiber uses at most $\lfloor b\rfloor\le b$ vertices, so a monochromatic copy in this blow-up would induce a monochromatic embedding into $\phi_m^{(\ell)}[b]$.

\section{Technical lemmas}\label{sec:technical}

We first recall two results of Brada\v{c}, Hunter and Sudakov~\cite{BHS}. The first provides bounded-degree graphs with robust neighborhood expansion into every large set, while the second supplies a sparse family of almost disjoint blocks whose intersections with large vertex sets are pseudorandomly distributed; together they underlie the expander component used in the reverse induction.

\begin{lemma}[\cite{BHS}]\label{lem:BHS-graph}
For any integer $k\ge1$ and any $\varepsilon>0$, there are constants $d=d(k,\varepsilon)$ and $M_0=M_0(k,\varepsilon)$ such that for all $M\ge M_0$, there is a graph $F$ on $M$ vertices satisfying the following:
\begin{enumerate}
\item[(i)] $\Delta(F)\le d$.
\item[(ii)] For every $U\subseteq V(F)$ with $|U|\ge\varepsilon M$, the number of vertices having fewer than $k$ neighbors in $U$ is at most $\varepsilon M$.
\end{enumerate}
\end{lemma}

\begin{lemma}[\cite{BHS}]\label{lem:BHS-hypergraph}
For every $\varepsilon>0$, there is a constant $C=C(\varepsilon)$ such that for all integers $n,s$ satisfying $n\ge Cs^2$ and $s\ge C$, there is a nonempty $s$-uniform $n$-vertex hypergraph $T$ satisfying the following:
\begin{enumerate}
\item[(i)] $\Delta(T)\le C$.
\item[(ii)] $|e\cap e'|<\varepsilon s$ for all distinct $e,e'\in E(T)$.
\item[(iii)] For every $A\subseteq V(T)$ with $|A|\ge\varepsilon n$, there are at most $\varepsilon e(T)$ hyperedges $e\in E(T)$ such that $|e\cap A|-(|A|/n)s>\varepsilon s$.
\end{enumerate}
\end{lemma}

To use these two structural ingredients in the hypergraph setting, we now encode graphs as uniform hypergraphs and introduce the lower-uniformity projections that appear throughout the reverse induction.

\begin{definition}\label{def:HkF}
Let $F$ be a graph and let $\ell\ge3$. The $\ell$-graph $H^{(\ell)}(F)$ has vertex set $V(F)$, and an $\ell$-set $e$ is an edge if there exists $v\in e$ such that $F[e\setminus\{v\}]$ is connected.
\end{definition}

\begin{definition}\label{def:Hkfamily}
Given a collection of graphs $\mathcal F$, all with vertex sets contained in $[n]$, define the $\ell$-graph $H^{(\ell)}(\mathcal F)$ on $[n]$ by $E(H^{(\ell)}(\mathcal F))=\bigcup_{F\in\mathcal F}E(H^{(\ell)}(F))$.
\end{definition}

\begin{definition}\label{def:lower-uniformity}
Let $H$ be a $k$-graph and let $U,W_1,\ldots,W_r$ be pairwise disjoint subsets of $V(H)$, where $0\le r\le k-3$. The $(k-r)$-graph $H(U;W_1,\ldots,W_r)$ has vertex set $U$, and $S\in\binom U{k-r}$ is an edge if there are vertices $(w_1,\ldots,w_r)\in W_1\times\cdots\times W_r$ such that $S\cup\{w_1,\ldots,w_r\}\in E(H)$. When $r=0$, this means $H(U)=H[U]$.
\end{definition}

In particular, for pairwise disjoint $U,W_4,\ldots,W_k$, the $3$-graph $H(U;W_4,\ldots,W_k)$ has vertex set $U$, and $xyz$ is an edge precisely when there are $w_i\in W_i$ for $i=4,\ldots,k$ such that $\{x,y,z,w_4,\ldots,w_k\}\in E(H)$.

The remaining lemmas develop the two new ingredients of the proof: a sparse ordered base palette and a robust random matching construction. We begin with the palette.

\begin{lemma}\label{lem:ordered-palette}
Let $\eta>0$ and $c_0>0$ be fixed. For sufficiently large $m$, put $q=\lfloor m^\eta\rfloor$, $\rho=(\log m)^{-1/4}$, $\beta=\rho/8$, and $S=\lfloor\exp(c_0m\log\log m)\rfloor$. Then there is a coloring $\xi:\binom{[S]}2\to[q]$ such that the $3$-graph $R_\xi$ on $[S]$, defined by $ij\ell\in E(R_\xi)$ if and only if $\xi(ij)=\xi(i\ell)$ for $i<j<\ell$, satisfies
\[
        e(R_\xi[U])\le\beta\binom{|U|}{3}
\]
for every $U\subseteq[S]$ with $|U|\ge m/2$.
\end{lemma}

When inserting this palette into the stepping-up construction, we identify $[S]$ with $[0,S)$ by the natural order-preserving relabeling.

\begin{proof}
Choose all colors $\xi(ij)$ independently and uniformly from $[q]$. Fix $T=\{v_1<\cdots<v_t\}\subseteq[S]$ with $t\ge m/2$. For $1\le a\le t-2$ and $c\in[q]$, put
\[
        N_{a,c}=|\{b:a<b\le t,\ \xi(v_av_b)=c\}|,\qquad
        M_a=\max_{c\in[q]}N_{a,c},\qquad
        Z_a=\sum_{c=1}^q\binom{N_{a,c}}2.
\]
Then
\[
        e(R_\xi[T])=\sum_{a=1}^{t-2}Z_a.
\]
Set
$u=\left\lfloor\frac{\beta t}{100}\right\rfloor$
and call row $a$ exceptional if $M_a\ge u$. If row $a$ is not exceptional, then $Z_a \le\frac12M_a\sum_{c=1}^qN_{a,c}<\frac u2(t-a)$. Hence all nonexceptional rows contribute less than
\[
        \frac u2\sum_{a=1}^{t-2}(t-a)
        <\frac{ut^2}{4}
        \le\frac{\beta t^3}{400}.
\]
If $e(R_\xi[T])\ge\beta\binom t3$, then, since $\binom t3\ge t^3/12$ for $t\ge6$, the exceptional rows contribute more than $\beta t^3/13$. Since every row contributes less than $t^2/2$, there are at least $2\beta t/13$ exceptional rows.

Put $r=\left\lfloor\beta t/8\right\rfloor$. For sufficiently large $m$,
\[
        r\ge\frac{\beta t}{16},\qquad
        u\ge\frac{\beta t}{200}.
\]
Thus the event $e(R_\xi[T])\ge\beta\binom t3$ implies that at least $r$ rows are exceptional.

For every exceptional row $a$, there are a color $c_a\in[q]$ and a set $S_a\subseteq\{a+1,\ldots,t\}$ with $|S_a|=u$ such that
\[
        \xi(v_av_b)=c_a\qquad\text{for all }b\in S_a.
\]
Choose $r$ exceptional rows and one such certificate for each row. Distinct rows use disjoint pair-color variables. Therefore
\[
\mathbb P\left(e(R_\xi[T])\ge\beta\binom t3\right)
\le
\binom tr\left(q\binom tu\right)^rq^{-ru}.
\]
Using $\binom tr\le\left(\frac{et}{r}\right)^r$ and $\binom tu\le\left(\frac{et}{u}\right)^u$, we obtain
\[
\begin{aligned}
\log\mathbb P\left(e(R_\xi[T])\ge\beta\binom t3\right)
&\le
r\log\frac{et}{r}
+ru\log\frac{et}{u}
-r(u-1)\log q.
\end{aligned}
\]
Since $r\ge\beta t/16$, $u\ge\beta t/200$, and $u\ge2$ for sufficiently large $m$,
\[
\begin{aligned}
r\log\frac{et}{r}
&\le \frac{ru}{2}\log\frac{16e}{\beta},\\
ru\log\frac{et}{u}
&\le ru\log\frac{200e}{\beta}.
\end{aligned}
\]
Moreover, $\log q=\eta\log m+O(1)$, whereas
$\log(1/\beta)=\frac14\log\log m+O(1)$. Hence, for sufficiently large $m$,
\[
        \frac12\log\frac{16e}{\beta}
        +\log\frac{200e}{\beta}
        \le\frac14\log q.
\]
Since also $u-1\ge u/2$, it follows that
\[
\mathbb P\left(e(R_\xi[T])\ge\beta\binom t3\right)
\le
\exp\left(-\frac{ru}{4}\log q\right)
\le
\exp\left(-c_1t^2\sqrt{\log m}\right)
\]
for some constant $c_1=c_1(\eta)>0$.

Finally, using
$\binom St\le\left(eS/t\right)^t$
and $\log S\le c_0m\log\log m$, we obtain
\[
\begin{aligned}
&\mathbb P\left(\exists T\subseteq[S],\ |T|\ge m/2:\
        e(R_\xi[T])\ge\beta\binom{|T|}{3}\right)\\
&\qquad\le
\sum_{t=\lceil m/2\rceil}^S
\exp\left(
        t\bigl(c_0m\log\log m+1\bigr)
        -c_1t^2\sqrt{\log m}
        \right)
=o(1).
\end{aligned}
\]
Indeed, for $t\ge m/2$, the negative quadratic term dominates both the
linear term in the exponent and the additional factor coming from the at
most $S$ possible values of $t$, since
$\sqrt{\log m}\gg\log\log m$.
\end{proof}

The preceding lemma gives an unweighted density estimate on every sufficiently large vertex set. For the later capacity argument, we need a weighted form that remains valid for arbitrary fiber profiles.

\begin{lemma}\label{lem:weighted-palette}
Let $m$ be sufficiently large, let $0<\rho\le1$, and let $R$ be a $3$-graph on $[S]$ such that $e(R[U])\le(\rho/8)\binom{|U|}{3}$ for every $U\subseteq[S]$ with $|U|\ge m/2$. If $e^{-m/8}\le\rho/4$, then for every $w_1,\ldots,w_S\in[0,1]$ with $x:=\sum_{i=1}^Sw_i\ge m$,
\[
        \sum_{ij\ell\in E(R)}w_iw_jw_\ell
        \le\frac{\rho}{4}\binom{x}{3}.
\]
\end{lemma}

\begin{proof}
Choose a random set $X\subseteq[S]$ by including each $i$ independently with probability $w_i$, and write $Y=|X|$. Then
\[
        \mathbb E e(R[X])
        =\sum_{ij\ell\in E(R)}w_iw_jw_\ell.
\]
Let $\mathcal E=\{Y\ge m/2\}$. On $\mathcal E$, we have $e(R[X])\le(\rho/8)\binom Y3$, while on $\mathcal E^c$, $e(R[X])<m^3/48$. Consequently,
\[
        \mathbb E e(R[X])
        \le\frac{\rho}{8}\mathbb E\binom Y3
        +\frac{m^3}{48}\mathbb P(\mathcal E^c).
\] Now $\mathbb E\binom Y3=\sum_{i<j<\ell}w_iw_jw_\ell\le x^3/6\le(3/2)\binom x3$ for $x\ge9$. Moreover, since $x\ge m$, the Chernoff bound gives $\mathbb P(\mathcal E^c)\le\mathbb P(Y<x/2)\le e^{-x/8}\le e^{-m/8}$. Since $\binom x3\ge m^3/12$ for $x\ge m\ge6$,
\[
\begin{aligned}
\mathbb E e(R[X])
&\le
\frac{3\rho}{16}\binom x3
+\frac{e^{-m/8}}4\binom x3
\le
\frac{\rho}{4}\binom x3.
\end{aligned}
\]
\end{proof}

The ordered palette also yields the blue obstruction needed at uniformity $k$. We first show that the absence of a large blue clique survives the stepping-up recursion. Define $t_3=q+2$ and $t_\ell=2t_{\ell-1}-1$ for $4\le\ell\le k$. Thus
\begin{equation}\label{eq:t-explicit}
        t_\ell=2^{\ell-3}(q+1)+1.
\end{equation}

\begin{lemma}\label{lem:blue-clique}
For every $\ell\in[3,k]$, the coloring $\phi_m^{(\ell)}$ contains no blue copy of
$
        K_{t_\ell}^{(\ell)}
$
on pairwise distinct vertices.
\end{lemma}

\begin{proof}
We argue by induction on $\ell$. For $\ell=3$, suppose that
$x_1<\cdots<x_{q+2}$
span a blue $K_{q+2}^{(3)}$. Then, for every $2\le a<b\le q+2$, the triple
$x_1x_ax_b$ is blue, and hence
$\xi_m(x_1x_a)\ne\xi_m(x_1x_b).$
Thus
\[
        \xi_m(x_1x_2),\ldots,\xi_m(x_1x_{q+2})
\]
are $q+1$ pairwise distinct colors, contradicting the fact that $\xi_m$ uses only $q$ colors.

Now let $\ell\ge4$ and assume the statement holds for $\ell-1$. Suppose for contradiction that $X=\{x_1<\cdots<x_t\}$ spans a blue $K_t^{(\ell)}$, where $t=t_\ell=2t_{\ell-1}-1$. For $1\le i\le t-1$, put $\delta_i=\delta(x_i,x_{i+1})$.
Since $x_1<x_t$, Property~C implies that the maximum of
$\delta_1,\ldots,\delta_{t-1}$ is attained at a unique index; let $p$ be this index. Set
\[
        L=\{x_1,\ldots,x_p\},\qquad
        R=\{x_{p+1},\ldots,x_t\}.
\]
Since $|L|+|R|=2t_{\ell-1}-1$, one of $L,R$ has size at least $t_{\ell-1}$.

Suppose first that $|R|\ge t_{\ell-1}$. Since $q\to\infty$ with $m$, we may assume throughout that $t_{\ell-1}\ge\ell-1$, so all auxiliary $\ell$-sets chosen below exist. We claim that
\begin{equation}\label{eq:right-decreasing}
        \delta_p>\delta_{p+1}>\cdots>\delta_{t-1}.
\end{equation}
The first inequality follows from the choice of $p$. Fix $j\in[p+1,t-2]$. Choose an $\ell$-set $Y\subseteq X$ containing
$x_p,x_j,x_{j+1},x_{j+2}$
and with all remaining vertices chosen from $R$. In the increasing ordering of $Y$, the first induced $\delta$-value is $\delta_p$, which is the unique maximum of the induced $\delta$-sequence. If this sequence were not monotone, then its maximum would occur at the first position, and the definition of $\phi_m^{(\ell)}$ would color $Y$ red. Since $Y$ is blue, the sequence is monotone; as its first term is the unique maximum, it is strictly decreasing. The consecutive vertices $x_j,x_{j+1},x_{j+2}$ therefore give
$\delta_j>\delta_{j+1}.$
This proves~\eqref{eq:right-decreasing}.

Let
\[
        D_R=\{\delta_p,\delta_{p+1},\ldots,\delta_{t-1}\}.
\]
We show that $D_R$ spans a blue complete $(\ell-1)$-graph. Choose arbitrary indices
$p\le i_1<\cdots<i_{\ell-1}\le t-1.$
Consider the $\ell$ vertices
$x_{i_1},x_{i_2},\ldots,x_{i_{\ell-1}},x_{i_{\ell-1}+1}.$
By~\eqref{eq:right-decreasing} and Property~C, their induced $\delta$-sequence is exactly
$\delta_{i_1}>\delta_{i_2}>\cdots>\delta_{i_{\ell-1}}.$
The corresponding $\ell$-edge is blue, so the monotone branch of the stepping-up definition gives
$\phi_m^{(\ell-1)} \bigl(\{\delta_{i_1},\ldots,\delta_{i_{\ell-1}}\}\bigr)=\blue.$
Hence $D_R$ spans a blue $K_{|R|}^{(\ell-1)}$, contradicting $|R|\ge t_{\ell-1}$ and the induction hypothesis.

The case $|L|\ge t_{\ell-1}$ is symmetric. For completeness, the same argument, now choosing an $\ell$-set containing
$x_j,x_{j+1},x_{j+2},x_{p+1}$
with all remaining vertices in $L$, gives
$\delta_1<\delta_2<\cdots<\delta_p.$
For arbitrary
$1\le i_1<\cdots<i_{\ell-1}\le p,$
the $\ell$ vertices
$x_{i_1},x_{i_1+1},x_{i_2+1},\ldots,x_{i_{\ell-1}+1}$
have induced $\delta$-sequence
$\delta_{i_1}<\delta_{i_2}<\cdots<\delta_{i_{\ell-1}}.$
Thus $\{\delta_1,\ldots,\delta_p\}$ spans a blue $K_{|L|}^{(\ell-1)}$, again contradicting the induction hypothesis. This completes the induction.
\end{proof}

The preceding lemma gives the blue obstruction needed at uniformity $k$.

\begin{lemma}\label{lem:blue-blocker}
Let
$
        Q=(k-1)(t_k-1)+1.
$
Then the complete $k$-graph $K_Q^{(k)}$ admits no blue monochromatic embedding into $\phi_m^{(k)}[b]$, for any $b>0$.
\end{lemma}

\begin{proof}
First observe that for every $y\in[0,M_k(m))$,
\begin{equation}\label{eq:all-equal-red}
        \phi_m^{(k)}(\{y,\ldots,y\})=\red.
\end{equation}
Indeed, this is immediate from the definition when $k=3$; for $k\ge4$, all induced $\delta$-values are zero, so the monotone branch recursively reduces to the $3$-uniform base case.

Suppose that
$h:V(K_Q^{(k)})\to[0,M_k(m))$
is a blue monochromatic embedding. By~\eqref{eq:all-equal-red}, every fiber of $h$ contains at most $k-1$ vertices. Hence the image of $h$ contains at least
$\left\lceil\frac{Q}{k-1}\right\rceil=t_k$
distinct values. Choose one vertex from each of $t_k$ distinct fibers. Since $K_Q^{(k)}$ is complete, every $k$ chosen vertices form an edge, and hence their $k$ distinct images form a blue edge of $\phi_m^{(k)}$. Thus these $t_k$ images span a blue $K_{t_k}^{(k)}$, contradicting Lemma~\ref{lem:blue-clique}.
\end{proof}

To complement the deterministic blue obstruction, we now build the robust random component that excludes almost red embeddings at the terminal $3$-uniform level. A near-perfect $k$-uniform matching on $[n]$ is obtained from a permutation of $[n]$ by grouping the first $k\lfloor n/k\rfloor$ entries into consecutive blocks of size $k$.

The concentration input needed to make this construction uniform over all large set choices is the following bounded-differences estimate for random permutations.

\begin{lemma}[Permutation bounded differences]\label{lem:permutation-bd}
Let $\pi_1,\ldots,\pi_d$ be independent uniformly random permutations of $[n]$, and let $Z=f(\pi_1,\ldots,\pi_d)$. Assume that swapping two entries in any one of the permutations changes $Z$ by at most $c$. Then, for every $t>0$,
\[
        \mathbb P(Z\le \mathbb EZ-t)
        \le
        \exp\left(-\frac{t^2}{8dnc^2}\right).
\]
\end{lemma}

\begin{proof}
Expose the $d$ permutations one coordinate at a time and consider the associated Doob martingale. There are at most $dn$ exposure steps. If the value revealed at one step is changed while the previous history is kept fixed, the two conditional completions can be coupled by a single transposition in the corresponding permutation. Hence the relevant martingale difference is bounded in absolute value by $2c$. Azuma--Hoeffding therefore gives
\[
        \mathbb P(Z\le \mathbb EZ-t)
        \le
        \exp\left(-\frac{t^2}{8dnc^2}\right).
\]
\end{proof}

We can now combine the weighted palette estimate with the random matching construction.

\begin{lemma}\label{lem:robust-random}
Fix $k\ge3$ and $\alpha\in(0,1]$. Let $\xi_m$ be a pair-coloring supplied by Lemma~\ref{lem:ordered-palette}, and let $R_m$ be the red $3$-graph determined by this ordered coloring. If $c_0=c_0(k,\alpha)>0$ is sufficiently small, then for all sufficiently large positive integers $m$ and $n$ there exists a $k$-graph $H_R$ on $[n]$ with $\Delta(H_R)\le m$ such that the following holds. For every collection of pairwise disjoint sets $U,W_4,\ldots,W_k\subseteq[n]$ of size at least $\alpha n$, and every map $h:U\to[0,M_3(m))$ satisfying $|h^{-1}(y)|\le |U|/m$ for every $y$, there is an edge $\{x,y,z,w_4,\ldots,w_k\}\in E(H_R)$ with $x,y,z\in U$ and $w_i\in W_i$ such that
\[
        \phi_m^{(3)}(\{h(x),h(y),h(z)\})=\blue.
\]
For $k=3$, the sets $W_4,\ldots,W_k$ are absent.
\end{lemma}

\begin{proof}
Put $r=\lfloor n/k\rfloor$. Choose independently $m$ random near-perfect $k$-uniform matchings on $[n]$, each generated from an independent uniformly random permutation. Denote them by $\mathcal M_1,\ldots,\mathcal M_m$, and let
\[
        H_R=\mathcal M_1\cup\cdots\cup\mathcal M_m.
\]
Clearly $\Delta(H_R)\le m$.

Fix pairwise disjoint $U,W_4,\ldots,W_k$, each of size at least $\alpha n$. Call an edge of a matching \emph{active} if it contains exactly three vertices of $U$ and exactly one vertex of each $W_i$, $i=4,\ldots,k$. Let $X_j$ be the number of active edges in $\mathcal M_j$, and put $X=\sum_{j=1}^mX_j$.
A fixed block of a random matching is a uniformly random $k$-subset of $[n]$, so the probability that it is active is
\[
        \frac{\binom{|U|}{3}\prod_{i=4}^k|W_i|}{\binom nk}.
\]
For sufficiently large $n$, this probability is at least $p_0:=k!\alpha^k/24>0$. Since $r\ge n/(2k)$,
\[
        \mathbb E X\ge \mu_0mn,\qquad \mu_0:=\frac{p_0}{2k}.
\]
Swapping two entries in one of the $m$ underlying permutations changes $X$ by at most $2$. Applying Lemma~\ref{lem:permutation-bd} with $d=m$, $c=2$ and $t=\mathbb EX/2$, and using $\mathbb EX\ge\mu_0mn$, gives
\[
        \mathbb P\left(X<\frac12\mathbb EX\right)
        \le
        \exp(-c_1mn)
\]
for a constant $c_1=c_1(k,\alpha)>0$. Put $\tau=\mu_0/2$; thus $\mathbb P(X<\tau mn)\le e^{-c_1mn}$. There are at most $(k-1)^n$ choices for the ordered tuple of pairwise disjoint sets $U,W_4,\ldots,W_k$. Since $m$ is sufficiently large, a union bound shows that, with probability at least $3/4$, the following event $\mathcal E$ holds:
\[
        X\ge\tau mn
\]
for every such choice of $U,W_4,\ldots,W_k$.

We next show that, with positive probability, no map satisfying the fiber bound violates the conclusion of the lemma on the event $\mathcal E$. For fixed $U$, call a map
$h:U\to[0,M_3(m))$ \emph{capacity-respecting} if
\[
        |h^{-1}(y)|\le \frac{|U|}{m}
        \qquad\text{for every }y\in[0,M_3(m)).
\]

Fix $U,W_4,\ldots,W_k$ and a capacity-respecting map $h$. Write $u=|U|$, $a_y=|h^{-1}(y)|$, and $b=u/m$.
Define $\mathcal A_h\subseteq\binom U3$ by declaring $xyz\in\mathcal A_h$ if either $h(x),h(y),h(z)$ are not pairwise distinct, or they are pairwise distinct and $\phi_m^{(3)}(\{h(x),h(y),h(z)\})=\red$. The number of triples of the first type is at most
\[
\begin{aligned}
        u\sum_y\binom{a_y}{2}
        &\le
        \frac u2\left(\max_y a_y\right)\sum_ya_y
        \le
        \frac{u^3}{2m}
        \le
        \frac4m\binom u3.
\end{aligned}
\]
For the second type, put $w_y=a_y/b$. Since a capacity-respecting map on the nonempty set $U$ exists, necessarily $b=u/m\ge1$. Thus $0\le w_y\le1$ and $\sum_yw_y=m$.
By Lemmas~\ref{lem:ordered-palette} and~\ref{lem:weighted-palette},
\[
        \sum_{ij\ell\in E(R_m)}w_iw_jw_\ell
        \le\frac\rho4\binom m3.
\]
Consequently, the number of red transversal triples is at most $(\rho b^3/4)\binom m3\le(\rho/4)\binom u3$, because
\[
        b^3\binom m3
        =\frac{u(u-b)(u-2b)}6
        \le\binom u3.
\]
Since $4/m\le\rho/4$ for sufficiently large $m$,
\begin{equation}\label{eq:Ah-density}
        |\mathcal A_h|
        \le\frac\rho2\binom u3.
\end{equation}

Let $\mathcal K_h$ be the family of all active $k$-sets whose three vertices in $U$ form a member of $\mathcal A_h$. Then
\begin{equation}\label{eq:Kh-density}
        |\mathcal K_h|
        =|\mathcal A_h|\prod_{i=4}^k|W_i|
        \le\frac\rho2\binom nk.
\end{equation}

For the fixed sets $U,W_4,\ldots,W_k$, call a capacity-respecting map $h$ \emph{bad} if it is an almost red embedding of
$H_R(U;W_4,\ldots,W_k)$
into $\phi_m^{(3)}[u/m]$. Suppose that $\mathcal E$ holds and that $h$ is bad. Then every active edge in every matching $\mathcal M_j$ belongs to $\mathcal K_h$, and
$\sum_{j=1}^mX_j\ge\tau mn.$
Set $L=\lfloor n/(4k)\rfloor$.
Since $X_j\le n/k$ and $kL/n\ge1/5$ for sufficiently large $n$,
\[
        \sum_{j=1}^m\min\{X_j,L\}
        \ge\frac{kL}{n}\sum_{j=1}^mX_j
        \ge\frac{\tau mn}{5}.
\]
Hence, whenever $\mathcal E$ holds and $h$ is bad, the realized matchings determine a set $\mathcal P$ of exactly $r_0=\lfloor\tau mn/10\rfloor$ active block positions, with at most $L$ selected positions in each matching, such that every block indexed by $\mathcal P$ belongs to $\mathcal K_h$. Thus a bad map on $\mathcal E$ has a certificate consisting of the position set $\mathcal P$ together with the requirement that all its prescribed blocks lie in $\mathcal K_h$.

There are at most
$2^{mr}\le2^{mn/k}$
possible choices for $\mathcal P$. Fix one such position set and reveal its prescribed blocks successively, matching by matching. Within any fixed matching, before the next prescribed block is exposed, fewer than $L$ prescribed blocks from that matching have already been revealed, so at least
$n-kL\ge\frac{3n}{4}$
vertices remain. Conditional on all previously exposed prescribed blocks, the next prescribed block position is a uniformly random $k$-subset of the remaining vertices of its matching. By~\eqref{eq:Kh-density}, its conditional probability of belonging to $\mathcal K_h$ is at most
\[
        \frac{(\rho/2)\binom nk}{\binom{3n/4}{k}}
        \le C_1\rho,
\]
where $C_1=C_1(k)$. Therefore the probability that all $r_0$ prescribed blocks belong to $\mathcal K_h$ is at most
$(C_1\rho)^{r_0}.$

The number of choices of $U,W_4,\ldots,W_k$ is at most $(k-1)^n$, and the number of maps $h:U\to[0,M_3(m))$ is at most $M_3(m)^n$. Let $P_{\mathrm{bad}}$ denote the probability that $\mathcal E$ holds and some capacity-respecting map $h$ is an almost red embedding. Then
\[
P_{\mathrm{bad}}
\le
(k-1)^nM_3(m)^n2^{mn/k}(C_1\rho)^{r_0}.
\]
For sufficiently large $m$,
$\log(C_1\rho)\le-\frac18\log\log m$
and
$r_0\ge\frac{\tau mn}{20}.$
Since $\log M_3(m)\le c_0m\log\log m,$
taking logarithms in the preceding bound gives
\begin{align*}
\log P_{\mathrm{bad}}
&\le n\log(k-1)+n\log M_3(m)+\frac{mn}{k}\log2+r_0\log(C_1\rho)\\
&\le n\log(k-1)+c_0mn\log\log m+\frac{mn}{k}\log2
      -\frac{\tau mn}{160}\log\log m.
\end{align*}
It follows that
\[
        \log P_{\mathrm{bad}}
        \le
        \left(c_0-\frac{\tau}{160}+o(1)\right)
        mn\log\log m.
\]
Thus, for example, choosing $0<c_0<\tau/320$ gives $P_{\mathrm{bad}}<1/4$ for all sufficiently large $m$ and $n$.

Combining this with $\mathbb P(\mathcal E^c)<1/4$, we obtain a realization of $H_R$ with the required property.
\end{proof}

\section{Proofs of the main theorems}\label{sec:proofs}

\medskip
\noindent\textbf{Proof strategy.}
The proof uses the stepping-up scheme of Brada\v{c}, Hunter and Sudakov~\cite{BHS}, but the $3$-uniform base coloring and the random part of the bounded-degree hypergraph are replaced by more robust objects. At the base level, we choose a $q$-coloring $\xi$ of the pairs of a set of size $\exp(\Theta_k(m\log\log m))$ and color a triple $i<j<\ell$ red precisely when $\xi(ij)=\xi(i\ell)$. The resulting red $3$-graph is uniformly sparse in a weighted sense, while the blue subhypergraph contains no $K_{q+2}^{(3)}$.

The random part $H_R$ is the union of $m$ independent random near-perfect $k$-uniform matchings. The key probabilistic step is to separate two sources of failure. With high probability, every choice of large disjoint sets $U,W_4,\ldots,W_k$ has $\Theta(mn)$ active edges across the matching layers. Conditioned on this event, an almost red embedding under any fixed capacity-respecting map would force $\Theta(mn)$ prescribed matching blocks to lie in a family of density $O((\log m)^{-1/4})$. Sequential exposure bounds the probability of this event by $\exp(-\Theta_k(mn\log\log m))$, which is sufficient for a union bound over all maps into the enlarged base palette.

The construction has three components. The robust random part $H_R$ is preserved throughout the reverse induction and, at uniformity $3$, excludes an almost red embedding by the weighted sparsity of the ordered palette. The expander part $H_E$ enforces the structure required by the reverse induction. Finally, a complete $k$-graph $H_B$ on $O_k(q)$ vertices excludes the blue case directly at uniformity $k$. This follows from a blue-clique stepping-up lemma: the base coloring contains no blue $K_{q+2}^{(3)}$, and this forbidden clique propagates through the stepping-up recursion with only a constant-factor loss in its order at each level.

With these ingredients in place, we first prove the $3$-uniform base case and then turn to the higher-uniformity theorem.

\begin{proof}[Proof of Theorem~\ref{thm:base-three}]
Set $k=3$, so $\eta=1/12$. Choose $c_0>0$ sufficiently small so that Lemma~\ref{lem:robust-random} applies with $\alpha=1$, and fix the corresponding ordered palette $\xi_m$ from Lemma~\ref{lem:ordered-palette}. Put $m=\lfloor\Delta/2\rfloor$.

By Lemma~\ref{lem:robust-random}, there is a $3$-graph $H_R$ on $[n]$ with $\Delta(H_R)\le m$ such that every map $h:[n]\to[0,M_3(m))$ with $|h^{-1}(y)|\le n/m$ admits an edge $xyz\in E(H_R)$ satisfying $\phi_m^{(3)}(\{h(x),h(y),h(z)\})=\blue$. Let $H_B$ be a copy of $K_{q+2}^{(3)}$ on a subset of $[n]$, and set $H=H_R\cup H_B$. Since $\Delta(H_B)=\binom{q+1}{2}=o(m)$, we have $\Delta(H_B)\le m$ for all sufficiently large $\Delta$, and hence
\[
        \Delta(H)\le\Delta(H_R)+\Delta(H_B)\le2m\le\Delta.
\]

We claim that $H$ has no monochromatic embedding into $\phi_m^{(3)}[n/m]$. A red embedding would contradict the defining property of $H_R$. If there were a blue embedding of $H_B$, then no two vertices of $H_B$ could have the same image, because every triple with repeated images is red by the definition of $\phi_m^{(3)}$. Hence the $q+2$ distinct images would span a blue $K_{q+2}^{(3)}$, contradicting Lemma~\ref{lem:blue-clique} in the base case $\ell=3$.

Therefore
\[
        r(H)>M_3(m)\left\lfloor\frac nm\right\rfloor.
\]
For all sufficiently large $\Delta$, we have $m\ge\Delta/3$, $\log\log m=(1+o(1))\log\log\Delta$, and $\lfloor n/m\rfloor\ge n/(2m)$. Also, $M_3(m)\ge\frac12\exp(c_0m\log\log m)$. Hence, after choosing a sufficiently small absolute constant $c>0$,
\[
        M_3(m)\left\lfloor\frac nm\right\rfloor
        \ge \frac{n}{4m}\exp(c_0m\log\log m)
        \ge 2^{c\Delta\log\log\Delta}\,n.
\]
Therefore $r(H)\ge 2^{c\Delta\log\log\Delta}n$ for all sufficiently large $\Delta$. This proves the theorem.
\end{proof}

We now pass from the base case to general $k$ by adding an expander component and a blue obstruction, and then propagating the red embedding through the reverse induction.

\begin{proof}[Proof of Theorem~\ref{main}]
The case $k=3$ is Theorem~\ref{thm:base-three}, so assume $k\ge4$. Set
\[
        \varepsilon=10^{-4k},\qquad
        \alpha_\ell=(10^3)^{-k+\ell}\quad(3\le\ell\le k),
        \qquad
        u_\ell=\lceil\alpha_\ell n\rceil.
\]
Thus $u_k=n$. We use ceilings for the set sizes in the reverse induction; the constants below have sufficient slack to absorb the resulting additive rounding errors.
Choose $c_0=c_0(k)>0$ sufficiently small so that Lemma~\ref{lem:robust-random} applies with $\alpha=\varepsilon$, and fix the base coloring $\xi_m$ supplied by Lemma~\ref{lem:ordered-palette}. Let $C_k$ be a sufficiently large constant depending only on $k$, and put $m=\lfloor\Delta/C_k\rfloor$.
For sufficiently large $\Delta$, we have $m\to\infty$ and $m\ge\Delta/(2C_k)$. Since $n\ge2^\Delta$, all lower bounds on $n$ required below hold for sufficiently large $\Delta$.

We construct three $k$-graphs on the common vertex set $[n]$.

\medskip
\noindent\emph{The robust random part.}
By Lemma~\ref{lem:robust-random}, there is a $k$-graph $H_R$ on $[n]$ with $\Delta(H_R)\le m$ satisfying the conclusion of that lemma for sets of size at least $\varepsilon n$.

\medskip
\noindent\emph{The expander part.}
Put $s=10^{6k}m$. Let $F$ be an $s$-vertex graph supplied by Lemma~\ref{lem:BHS-graph} with parameters $k$ and $\varepsilon$; write $\Delta(F)\le d_0$, where $d_0=d_0(k)$ is independent of $m$. Let $T$ be an $s$-uniform hypergraph on $[n]$ supplied by Lemma~\ref{lem:BHS-hypergraph} with parameter $\varepsilon$, and write $E(T)=\{B_1,\ldots,B_{e(T)}\}$.
For each $i\in[e(T)]$, let $F_i$ be an isomorphic copy of $F$ with vertex set $B_i$, and put $\mathcal F=\{F_i:i\in[e(T)]\}$ and $H_E=H^{(k)}(\mathcal F)$.
For a fixed vertex $x\in V(F_i)$, the number of connected $(k-1)$-subsets of $F_i$ containing $x$ is at most
\[
        (k-2)!\,d_0^{k-2}.
\]
Indeed, every such set admits an ordering
$x=v_1,v_2,\ldots,v_{k-1}$ such that, for each $j\ge2$, the vertex
$v_j$ has a neighbor among $v_1,\ldots,v_{j-1}$. Once
$v_1,\ldots,v_{j-1}$ have been chosen, there are at most
$(j-1)d_0$ choices for $v_j$, and hence the number of possible ordered
explorations is at most
\[
        \prod_{j=2}^{k-1}(j-1)d_0
        =(k-2)!\,d_0^{k-2}.
\]
Fix $x\in B_i$. If $x$ is the exceptional vertex in
Definition~\ref{def:HkF}, choose one vertex of the connected part as a
root; there are at most $s$ choices for this root and then at most
$(k-2)!\,d_0^{k-2}$ choices for the connected part. If $x$ lies in the
connected part, there are at most $s$ choices for the exceptional vertex
and at most $(k-2)!\,d_0^{k-2}$ choices for the connected part. Hence
\[
        \Delta(H^{(k)}(F_i))
        \le 2(k-2)!\,s d_0^{k-2}
        =O_k(s).
\]
Since $\Delta(T)=O_k(1)$,
\begin{equation}\label{eq:HE-degree}
        \Delta(H_E)=O_k(s)=O_k(m).
\end{equation}

\medskip
\noindent\emph{The blue obstruction.}
Recall the integers $t_3,\ldots,t_k$ defined before Lemma~\ref{lem:blue-clique}, and set $Q=(k-1)(t_k-1)+1$.
Choose any $Q$-vertex subset of $[n]$ and let $H_B$ be the complete $k$-graph on this set. By~\eqref{eq:t-explicit},
$Q=O_k(q),$
and hence
\begin{equation}\label{eq:HB-degree}
        \Delta(H_B)=\binom{Q-1}{k-1}
        =O_k(q^{k-1})
        =o(m).
\end{equation}

Finally, let $$H=H_R\cup H_E\cup H_B.$$ By~\eqref{eq:HE-degree} and~\eqref{eq:HB-degree}, choosing $C_k$ sufficiently large gives $\Delta(H)\le\Delta$.
We shall prove that there is no monochromatic embedding of $H$ into $\phi_m^{(k)}[b_k]$, where $b_k=(10^3)^{-2k+6}n/m$.
The definition of $M_3(m)$ and the recursion~\eqref{eq:M-recursion} imply that, for some constant $a_k>0$ and all sufficiently large $m$,
\[
        M_k(m)\ge\tw_{k-1}\!\bigl(a_km\log\log m\bigr).
\]
Indeed, after decreasing $a_3>0$ if necessary, we have $M_3(m)\ge\tw_2(a_3m\log\log m)$. If $M_\ell(m)\ge\tw_{\ell-1}(a_\ell m\log\log m)$, then for any fixed $0<a_{\ell+1}<a_\ell$ and all sufficiently large $m$, $M_\ell(m)-1\ge\tw_{\ell-1}(a_{\ell+1}m\log\log m)$, and hence $M_{\ell+1}(m)=2^{M_\ell(m)-1}\ge\tw_\ell(a_{\ell+1}m\log\log m)$.
We next make the final tower comparison explicit.  Put
$A=a_km\log\log m$.  Since $m=\lfloor\Delta/C_k\rfloor$, after choosing
$c_k>0$ sufficiently small we have
$c_k\Delta\log\log\Delta\le A/2$ for all sufficiently large $\Delta$.
Moreover, since $k\ge4$,
\[
 \frac{\tw_{k-1}(A)}{\tw_{k-1}(A/2)m}
 =\frac{2^{\tw_{k-2}(A)-\tw_{k-2}(A/2)}}{m}
 \longrightarrow\infty
 \qquad (\Delta\to\infty).
\]  Since
$b_k=(10^3)^{-2k+6}n/m\to\infty$, we also have
$\lfloor b_k\rfloor\ge b_k/2$.  Hence, for all sufficiently large
$\Delta$,
\begin{equation}\label{eq:final-tower-comparison}
\begin{aligned}
 M_k(m)\lfloor b_k\rfloor
 &\ge \tw_{k-1}(A)\frac{(10^3)^{-2k+6}n}{2m}\\
 &>\tw_{k-1}(A/2)n
 \ge\tw_{k-1}\!\bigl(c_k\Delta\log\log\Delta\bigr)n.
\end{aligned}
\end{equation}
Thus it suffices to rule out a monochromatic embedding of $H$ into
$\phi_m^{(k)}[b_k]$.

Suppose for contradiction that such an embedding exists. If its color is blue, then its restriction to $H_B$ is a blue monochromatic embedding of $K_Q^{(k)}$ into $\phi_m^{(k)}[b_k]$, contradicting Lemma~\ref{lem:blue-blocker}. Hence the original embedding is red. From now on, all embeddings in the reverse induction are red. To rule out this remaining case, we decrease the uniformity one step at a time. The following claim records the sets, pruned blocks, and embeddings preserved at each stage.

\begin{claim}\label{claim:induction}
Let $\ell\in[3,k]$. Then there are pairwise disjoint sets
$U^\ell,W_{\ell+1},\ldots,W_k\subseteq[n]$
and sets
$B_1^\ell,\ldots,B_{e(T)}^\ell$
with $B_i^\ell\subseteq B_i\cap U^\ell$ satisfying the following.

\medskip
\noindent\textup{(i)} $|U^\ell|=u_\ell$ and $|W_i|=u_{i-1}$ for every $i\in[\ell+1,k]$.

\medskip
\noindent\textup{(ii)} For every $i\in[e(T)]$, either $B_i^\ell=\varnothing$, or
$|B_i\cap U^\ell|\ge(\alpha_\ell-\varepsilon)s$
and
$|(B_i\setminus B_i^\ell)\cap U^\ell| \le(k-\ell)\varepsilon s.$
Furthermore,
$|\{i:B_i^\ell=\varnothing\}| \le3(k-\ell)\varepsilon e(T).$

\medskip
\noindent\textup{(iii)} Write $H_R^\ell=H_R(U^\ell;W_{\ell+1},\ldots,W_k)$ and
\[
        H_E^\ell=H^{(\ell)}\bigl(\{F_i[B_i^\ell]:i\in[e(T)]\}\bigr).
\]
Also set $b_\ell=(10^3)^{-k-\ell+6}n/m$ and $\phi^\ell=\phi_m^{(\ell)}$. Then there is a map $h^\ell:U^\ell\to[0,M_\ell(m))$ which is an almost red embedding of $H_R^\ell$ into $\phi^\ell[b_\ell]$ and a red monochromatic embedding of $H_E^\ell$ into $\phi^\ell[b_\ell]$.
\end{claim}

\begin{proof}
We prove the claim by reverse induction on $\ell$.

For $\ell=k$, take $U^k=[n]$ and $B_i^k=B_i$ for every $i\in[e(T)]$, and let $h^k$ be the original red monochromatic embedding of $H$. Then $H_R^k=H_R$ and $H_E^k=H_E$, so condition~\textup{(iii)} holds; conditions~\textup{(i)} and~\textup{(ii)} are immediate.

Now let $4\le\ell\le k$, and suppose the claim holds at level $\ell$. Put $n_\ell=|U^\ell|=u_\ell$, so $n_\ell\ge\alpha_\ell n$. After applying a common relabeling of the ambient vertex set and all the objects constructed above, we may assume that $U^\ell=[n_\ell]$ and
\[
        h^\ell(i)=x_i,\qquad
        0\le x_1\le x_2\le\cdots\le x_{n_\ell}<M_\ell(m).
\]
For $1\le i<n_\ell$, write $\delta_i=\delta(x_i,x_{i+1})$.

We first isolate two well-separated intervals. Starting with $L_1=1$ and $R_1=n_\ell$, construct nested intervals $[L_j,R_j]$ as follows. If
$R_j-L_j<n_\ell/50,$
stop. Otherwise the interval $[L_j,R_j]$ contains more than $n_\ell/50$ vertices. Since
$b_\ell<\frac{n_\ell}{50}$
for sufficiently large $m$, the fiber bound on $h^\ell$ implies that $x_{L_j},\ldots,x_{R_j}$ contain at least two distinct values. In particular, $x_{L_j}<x_{R_j}$, so Property~C gives a unique index
$p_j\in[L_j,R_j-1]$
at which $\delta_i$ is maximal on $[L_j,R_j-1]$; moreover,
\begin{equation}\label{eq:positive-stage-max}
        \delta_{p_j}=\delta(x_{L_j},x_{R_j})>0.
\end{equation}
The two candidate retained intervals are $[L_j,p_j]$ and
$[p_j+1,R_j]$. Retain one having maximum cardinality: if one is larger,
retain that one, while in the case of equal cardinalities either choice is
allowed. We call $j$ a left step when $[L_j,p_j]$ is retained, and a right
step when $[p_j+1,R_j]$ is retained. Thus, at every step, the retained
interval has at least half the cardinality of the preceding interval, and
the final interval has cardinality at least $n_\ell/100$.

We use the following elementary consequence of this procedure. There exist
$1<t_1<t_2$
such that either
\begin{equation}\label{eq:right-case}
\begin{aligned}
&t_1-1,t_2-1\text{ are right steps},\\
&L_{t_1}-L_1,\quad L_{t_2}-L_{t_1}
        \ge n_\ell/64,
\end{aligned}
\end{equation}
or the symmetric statement holds with left steps and the corresponding decreases of the right endpoints.

Indeed, put $t_0'=1$, $a_0=1$, and $b_0=n_\ell$. For $r\in [3]$, let $t_r'$ be the first stage after $t_{r-1}'$ at which either $L_{t_r'}-a_{r-1}\ge n_\ell/64$ or $b_{r-1}-R_{t_r'}\ge n_\ell/64$, and then set $a_r=L_{t_r'}$ and $b_r=R_{t_r'}$.
We verify inductively that these three stages exist and that
\begin{equation}\label{eq:quarter-retention}
        b_r-a_r+1\ge\frac{b_{r-1}-a_{r-1}+1}{4}.
\end{equation}
Put $D=n_\ell/64$ and
$N_{r-1}=b_{r-1}-a_{r-1}+1$. Before either endpoint displacement first reaches
$D$, both displacements are smaller than $D$. Hence the current interval has
cardinality at least $N_{r-1}-2D$. The step at which the threshold is
first reached retains at least half of that interval, and therefore
\[
        b_r-a_r+1
        \ge\frac{N_{r-1}-2D}{2}.
\]
Inductively, $N_{r-1}\ge n_\ell/4^{r-1}\ge n_\ell/16=4D$, so the
last display is at least $N_{r-1}/4$, proving~\eqref{eq:quarter-retention}.
The process also cannot stop before one endpoint displacement reaches $D$:
otherwise its terminal interval would still have cardinality at least
$N_{r-1}-2D\ge n_\ell/32$, whereas the stopping rule gives cardinality
at most $n_\ell/50+1$. For sufficiently large $n_\ell$ this is a
contradiction. Thus all three $t_r'$ are well-defined.

If $a_r-a_{r-1}\ge n_\ell/64$, then $t_r'-1$ is a right step; if $b_{r-1}-b_r\ge n_\ell/64$, then $t_r'-1$ is a left step. Among the three indices $r\in[3]$, at least two correspond to the same side. Taking the two corresponding stages gives~\eqref{eq:right-case} or its left-right mirror image.

The two alternatives are exact mirror images. We record the symmetry explicitly. Since $\ell\ge4$, $M_\ell(m)=2^{M_{\ell-1}(m)-1}$. Define
\[
        \theta_\ell(x)=M_\ell(m)-1-x
        \qquad(0\le x<M_\ell(m)).
\]
This is bitwise complementation in the first $M_{\ell-1}(m)-1$ binary positions. Hence
$\delta(\theta_\ell(x),\theta_\ell(y))=\delta(x,y)$
for all $x,y$, while $\theta_\ell$ reverses the numerical order. Consequently, if
$x_1\le\cdots\le x_\ell,$
then the induced $\delta$-sequence of
$\theta_\ell(x_\ell)\le\cdots\le\theta_\ell(x_1)$
is the reverse of the original one. The monotone branch therefore passes exactly the same multiset of $\delta$-values to level $\ell-1$, and in the nonmonotone branch the unique maximum moves from position $i$ to position $\ell-i$; in particular, endpoint maxima remain endpoint maxima and interior maxima remain interior maxima. Thus
\begin{equation}\label{eq:reflection-invariance}
        \phi_m^{(\ell)}
        (\{\theta_\ell(x_1),\ldots,\theta_\ell(x_\ell)\})
        =
        \phi_m^{(\ell)}(\{x_1,\ldots,x_\ell\}).
\end{equation}
Therefore, after replacing $h^\ell$ by $\theta_\ell\circ h^\ell$ and reversing the ordering of $U^\ell$ (together with the corresponding common relabeling of all ambient objects), the ordered image sequence becomes $\theta_\ell(x_{n_\ell})\le\cdots\le\theta_\ell(x_1)$. This preserves every fiber size and both red embedding properties in condition~\textup{(iii)}. The reflected nested intervals form an admissible interval process for the reversed sequence: whenever the two candidate sides have unequal cardinalities, the larger side is reflected to the larger side, and at a tie we choose the reflected side, as permitted by the construction. Thus every left step is reflected to a right step, and all endpoint separations are preserved. Hence the left-step alternative reduces to the right-step alternative by~\eqref{eq:reflection-invariance}. It is therefore enough to prove~\eqref{eq:right-case}.

Set
\[
        I_1=[1,L_{t_1}-1],\qquad
        I_2=[L_{t_1},L_{t_2}-1].
\]
Then $|I_1|,|I_2|\ge n_\ell/64$.
Let $\mathcal R\subseteq[t_1-1]$ be the set of right steps. For $j\in\mathcal R$, put
$S_j=[L_j,p_j].$
The sets $S_j$ are pairwise disjoint, appear in increasing order, and satisfy
\begin{equation}\label{eq:I1-partition}
        I_1=\bigcup_{j\in\mathcal R}S_j.
\end{equation}

For $A\subseteq[n]$ and $i\in[e(T)]$, say that $B_i$ is \emph{deficient on $A$} if
\[
        |A\cap B_i|<\frac{|A|}{n}s-\varepsilon s.
\]
We shall apply this notion only to $I_1,I_2$, to the union $S$ introduced below, and to $U^{\ell-1}$. Each of these sets has size at least $\varepsilon n$, and its complement also has size at least $\varepsilon n$: for the intervals and the union $S$ this follows from the disjoint linear-sized intervals among $I_1,I_2$, while for $U^{\ell-1}$ it follows, for example, from the disjoint set $I_2$. Since
$|([n]\setminus A)\cap B_i|-\frac{|[n]\setminus A|}{n}s =-\left(|A\cap B_i|-\frac{|A|}{n}s\right),$
Lemma~\ref{lem:BHS-hypergraph}\textup{(iii)}, applied to $[n]\setminus A$, shows that for each such $A$ at most $\varepsilon e(T)$ indices $i$ are deficient on $A$.

We now prune the sets $B_i^\ell$. If $B_i^\ell=\varnothing$, or if $B_i$ is deficient on at least one of $I_1,I_2$, set $B_i'=\varnothing$. Otherwise, let $B_i'$ consist of the vertices $v\in B_i^\ell\cap I_1$ having at least $k$ neighbors in $B_i^\ell\cap I_2$ in the graph $F_i$.

Suppose $B_i^\ell\ne\varnothing$ and $B_i$ is not deficient on any $I_j$. Then, for $j\in[2]$,
\[
\begin{aligned}
|B_i^\ell\cap I_j|
&\ge |B_i\cap I_j|-|(B_i\setminus B_i^\ell)\cap U^\ell|\\
&\ge \frac{|I_j|}{n}s-\varepsilon s-(k-\ell)\varepsilon s
>\varepsilon s,
\end{aligned}
\]
where the last inequality follows from $|I_j|\ge\alpha_\ell n/64$ and the choices of $\alpha_\ell,\varepsilon$. Lemma~\ref{lem:BHS-graph}\textup{(ii)} therefore shows that at most $\varepsilon s$ vertices of $B_i^\ell\cap I_1$ fail the neighborhood condition defining $B_i'$. By condition~\textup{(ii)} of the induction hypothesis, at most $(k-\ell)\varepsilon s$ vertices of $B_i\cap U^\ell$ were already removed when $B_i^\ell$ was formed. Adding these two losses gives
\begin{equation}\label{eq:pruning-loss}
        |(B_i\setminus B_i')\cap I_1|
        \le(k-\ell+1)\varepsilon s.
\end{equation}
Moreover,
\[
        |B_i'|
        \ge\frac{|I_1|}{n}s-\varepsilon s-(k-\ell+1)\varepsilon s
        >\frac{\alpha_\ell s}{100}
        >\varepsilon s.
\]
A set $B_i'$ can be empty only if $B_i^\ell=\varnothing$ or if $B_i$ is deficient on at least one of $I_1,I_2$. The first alternative occurs for at most $3(k-\ell)\varepsilon e(T)$ indices by the induction hypothesis, and each of the two deficiency conditions occurs for at most $\varepsilon e(T)$ indices. Hence
\begin{equation}\label{eq:empty-Biprime}
        |\{i:B_i'=\varnothing\}|
        \le(3(k-\ell)+2)\varepsilon e(T).
\end{equation}

We next show that no set $B_i'$ can occupy two vertices of the same interval $S_j$. The following schematic records the interval configuration used in the argument.

\begin{center}
\begin{tikzpicture}[x=0.8cm,y=0.8cm]
  \draw[->] (0,0)--(13,0);
  \draw[very thick] (1.2,0)--(4.4,0);
  \draw[very thick] (7.0,0)--(11.6,0);
  \draw (1.2,0.12)--(1.2,-0.12) node[below] {$L_j$};
  \draw (4.4,0.12)--(4.4,-0.12) node[below] {$p_j$};
  \draw (7.0,0.12)--(7.0,-0.12) node[below] {$L_{t_1}$};
  \draw (12.0,0.12)--(12.0,-0.12) node[below] {$L_{t_2}-1$};
  \node[above] at (2.8,0.15) {$S_j$};
  \node[above] at (9.5,0.15) {$I_2$};
  \fill (2.0,0) circle (1.6pt) node[below=4pt] {$v$};
  \fill (3.2,0) circle (1.6pt) node[below=4pt] {$w$};
  \fill (8.1,0) circle (1.6pt) node[below=4pt] {$z_1$};
  \node[below=3pt] at (9.3,0) {$\cdots$};
  \fill (10.2,0) circle (1.6pt) node[below=4pt] {$z_{\ell-2}$};
  \draw[->] (4.4,1.25)--(4.4,0.22);
  \node[above] at (4.4,1.25) {$\delta_{p_j}\text{ is the unique maximum}$};
\end{tikzpicture}
\end{center}

\medskip
\noindent\emph{Fact.}
For every $i\in[e(T)]$ and every $j\in\mathcal R$,
$|B_i'\cap S_j|\le1.$

\smallskip
Suppose instead that $v<w$ are two vertices of $B_i'\cap S_j$. Since $v\in B_i'$, choose distinct vertices
$z_1<\cdots<z_{\ell-2}$
from $B_i^\ell\cap I_2$ that are all adjacent to $v$ in $F_i$. Then
$e=\{v,w,z_1,\ldots,z_{\ell-2}\}$
is an edge of $H_E^\ell$: the $\ell-1$ vertices $v,z_1,\ldots,z_{\ell-2}$ induce a connected star in $F_i$.

By the interval construction, $\delta_{p_j}$ is the unique maximum on $[L_j,R_j-1]$. Since $j<t_1$ and $j$ is a right step, every later retained interval lies in $[p_j+1,R_j]$. In particular, $I_2\subseteq[p_j+1,R_j]$. Hence
\[
        \delta(x_v,x_w)<\delta_{p_j},
        \qquad
        \delta(x_w,x_z)=\delta_{p_j}
        \quad\text{for every }z\in I_2.
\]
Moreover, if $z<z'$ lie in $I_2$, then every consecutive $\delta$-value between $z$ and $z'$ is strictly smaller than $\delta_{p_j}$; therefore Property~C gives
\[
        \delta(x_z,x_{z'})<\delta_{p_j}.
\]
Consequently,
$\delta(x_v,x_w)<\delta_{p_j}=\delta(x_w,x_{z_1}),$
while
$\delta(x_{z_r},x_{z_{r+1}})<\delta_{p_j}$ for every $1\le r\le\ell-3$.
Thus the induced $\delta$-sequence of
$x_v,x_w,x_{z_1},\ldots,x_{z_{\ell-2}}$
has its unique maximum at the second position. By the definition of $\phi_m^{(\ell)}$, the edge $e$ is blue, contradicting the fact that $h^\ell$ is a red monochromatic embedding of $H_E^\ell$. This proves the fact.

Now put
\[
        \mathcal R_S=\{j\in\mathcal R:|S_j|\le b_{\ell-1}\},\qquad
        \mathcal R_B=\mathcal R\setminus\mathcal R_S.
\]
We claim that
\begin{equation}\label{eq:small-union}
        \left|\bigcup_{j\in\mathcal R_S}S_j\right|
        \ge\frac{n_\ell}{200}.
\end{equation}
Suppose otherwise. Since $|I_1|\ge n_\ell/64$,~\eqref{eq:I1-partition} gives
$\left|\bigcup_{j\in\mathcal R_B}S_j\right| \ge\frac{n_\ell}{200}.$
Let
$S=\bigcup_{j\in\mathcal R_B}S_j.$
Since $|S|\ge\varepsilon n$, at most $\varepsilon e(T)$ indices
$i$ are deficient on $S$.  Together with~\eqref{eq:empty-Biprime}, the
number of indices for which either $B_i'=\varnothing$ or $B_i$ is
deficient on $S$ is at most
\[
        (3(k-\ell)+3)\varepsilon e(T)<e(T),
\]
where the last inequality follows from $\varepsilon=10^{-4k}$.  Hence
there exists an index $i$ such that $B_i'\ne\varnothing$ and $B_i$ is
not deficient on $S$. By~\eqref{eq:pruning-loss},
\[
\begin{aligned}
|B_i'\cap S|
&\ge |B_i\cap S|-|(B_i\setminus B_i')\cap I_1|\\
&\ge\frac{|S|}{n}s-\varepsilon s-(k-\ell+1)\varepsilon s\\
&>\frac{\alpha_\ell s}{400}.
\end{aligned}
\]
On the other hand,
$|\mathcal R_B| \le\frac{n_\ell}{b_{\ell-1}} \le 2(10^3)^{2\ell-7}m$
for all sufficiently large $n$.
Since $s=10^{6k}m$, the preceding lower bound gives
$|B_i'\cap S|>|\mathcal R_B|$
for every $\ell\in[4,k]$. The pigeonhole principle then gives some $j\in\mathcal R_B$ with
$|B_i'\cap S_j|\ge2,$
contradicting the fact. This proves~\eqref{eq:small-union}.

Choose an arbitrary set $U^{\ell-1}\subseteq\bigcup_{j\in\mathcal R_S}S_j$ and an arbitrary set $W_\ell\subseteq I_2$ such that
\[
        |U^{\ell-1}|=|W_\ell|=u_{\ell-1}.
\]
Since $\alpha_{\ell-1}=10^{-3}\alpha_\ell$, we have
$u_{\ell-1}\le n_\ell/10^3+1$. Thus these choices are possible, for
sufficiently large $n$, by~\eqref{eq:small-union} and
$|I_2|\ge n_\ell/64$.

For $i\in[e(T)]$, define
\[
        B_i^{\ell-1}=
\begin{cases}
B_i'\cap U^{\ell-1},&
\text{if $B_i'\ne\varnothing$ and $B_i$ is not deficient on $U^{\ell-1}$},\\
\varnothing,&\text{otherwise}.
\end{cases}
\]
If $B_i'\ne\varnothing$ and $B_i$ is not deficient on $U^{\ell-1}$, then, since $U^{\ell-1}\subseteq I_1$,
\[
\begin{aligned}
|B_i'\cap U^{\ell-1}|
&\ge |B_i\cap U^{\ell-1}|-| (B_i\setminus B_i')\cap I_1|\\
&\ge (\alpha_{\ell-1}-\varepsilon)s-(k-\ell+1)\varepsilon s>0.
\end{aligned}
\]
Thus the set in the first case of the definition is nonempty. If $B_i^{\ell-1}\ne\varnothing$, then
$|B_i\cap U^{\ell-1}| \ge(\alpha_{\ell-1}-\varepsilon)s$
and, by~\eqref{eq:pruning-loss},
$|(B_i\setminus B_i^{\ell-1})\cap U^{\ell-1}| \le(k-\ell+1)\varepsilon s.$
Furthermore, by~\eqref{eq:empty-Biprime} and one more application of Lemma~\ref{lem:BHS-hypergraph}\textup{(iii)},
$|\{i:B_i^{\ell-1}=\varnothing\}| \le3(k-\ell+1)\varepsilon e(T).$
Thus conditions~\textup{(i)} and~\textup{(ii)} hold at level $\ell-1$.

It remains to construct the new embedding and verify the required color conditions. Put $j^*=L_{t_1}$ and $x^*=x_{j^*}$. For $v\in U^{\ell-1}\subseteq I_1$, define
\begin{equation}\label{eq:new-h}
        h^{\ell-1}(v)=\delta(x_v,x^*)
        =\max_{z\in[v,j^*)}\delta_z.
\end{equation}
Since $x_v,x^*<M_\ell(m)=2^{M_{\ell-1}(m)-1}$, the map $h^{\ell-1}$ takes values in $[0,M_{\ell-1}(m))$.

For $j\in\mathcal R_S$ and $v\in S_j$, the interval $[v,j^*)$ contains $p_j$, so $h^{\ell-1}(v)\ge\delta_{p_j}$. If $j'>j$ and $v'\in S_{j'}$, then $[v',j^*)$ is contained in the retained interval after the right step $j$ and therefore avoids the unique maximum $\delta_{p_j}$; all its $\delta$-values are strictly smaller than $\delta_{p_j}$. Consequently,
\begin{equation}\label{eq:strict-between-blocks}
        j<j',\ v\in S_j,\ v'\in S_{j'}
        \quad\Longrightarrow\quad
        h^{\ell-1}(v)>h^{\ell-1}(v').
\end{equation}
Since every $S_j$ with $j\in\mathcal R_S$ has size at most $b_{\ell-1}$,~\eqref{eq:strict-between-blocks} implies
\[
        |(h^{\ell-1})^{-1}(y)|\le b_{\ell-1}
        \qquad\text{for every }y.
\]

We now verify explicitly that $h^{\ell-1}$ is an almost red embedding of $H_R^{\ell-1}$. Let
$e=\{v_1<\cdots<v_{\ell-1}\}\in E(H_R^{\ell-1}).$
By the definition of lower uniformity, there is a vertex
$v_\ell\in W_\ell\subseteq I_2$
such that
$e\cup\{v_\ell\}\in E(H_R^\ell).$
The values in~\eqref{eq:new-h} are nonincreasing as the vertex index increases. If
$h^{\ell-1}(v_1),\ldots,h^{\ell-1}(v_{\ell-1})$
are not all distinct, there is nothing to prove, since we seek only an almost red embedding. Assume therefore that they are pairwise distinct. Then
\begin{equation}\label{eq:lower-strict-dec}
        h^{\ell-1}(v_1)>
        h^{\ell-1}(v_2)>
        \cdots>
        h^{\ell-1}(v_{\ell-1}).
\end{equation}
For $1\le r\le\ell-2$,~\eqref{eq:new-h} and~\eqref{eq:lower-strict-dec} imply that the maximum of $\delta_z$ on $[v_r,j^*)$ is already attained before $v_{r+1}$; hence
$\delta(x_{v_r},x_{v_{r+1}}) =h^{\ell-1}(v_r).$
Moreover, $t_1-1$ is a right step and
$p_{t_1-1}=L_{t_1}-1=j^*-1.$
By~\eqref{eq:positive-stage-max},
$\delta_{j^*-1}>0.$
Since $v_{\ell-1}<j^*$, equation~\eqref{eq:new-h} gives
$h^{\ell-1}(v_{\ell-1})\ge\delta_{j^*-1}>0.$
Also $v_\ell\in I_2$, and every $\delta_z$ with $z\in[j^*,L_{t_2})$ is strictly smaller than $\delta_{j^*-1}\le h^{\ell-1}(v_{\ell-1})$. Therefore
$\delta(x_{v_{\ell-1}},x_{v_\ell}) =h^{\ell-1}(v_{\ell-1})>0.$
Thus the induced $\delta$-sequence of
$x_{v_1},x_{v_2},\ldots,x_{v_{\ell-1}},x_{v_\ell}$
is exactly the strictly decreasing sequence
$h^{\ell-1}(v_1),\ldots,h^{\ell-1}(v_{\ell-1}).$
By~\eqref{eq:lower-strict-dec} and the positivity of the final term just proved, every term of this strictly decreasing induced $\delta$-sequence is positive. Consequently no two consecutive images are equal; since the images are nondecreasing with the vertex order, the $\ell$ images under $h^\ell$ are pairwise distinct. Since $h^\ell$ is an almost red embedding of $H_R^\ell$,
\[
        \phi_m^{(\ell)}
        (\{x_{v_1},\ldots,x_{v_\ell}\})=\red.
\]
The induced $\delta$-sequence is monotone, and therefore the stepping-up rule gives
\[
\begin{aligned}
\red
&=\phi_m^{(\ell)}
        (\{x_{v_1},\ldots,x_{v_\ell}\})\\
&=\phi_m^{(\ell-1)}
        (\{h^{\ell-1}(v_1),\ldots,h^{\ell-1}(v_{\ell-1})\}).
\end{aligned}
\]
This proves that $h^{\ell-1}$ is an almost red embedding of $H_R^{\ell-1}$.

Finally, let
$e=\{v_1<\cdots<v_{\ell-1}\}\in E(H_E^{\ell-1}).$
Then $e\subseteq B_i^{\ell-1}\subseteq B_i'$ for some $i$, and some $\ell-2$ vertices of $e$ induce a connected subgraph of $F_i$. Choose one vertex $v$ of this connected set. Since $v\in B_i'$, it has at least $k$ neighbors in $B_i^\ell\cap I_2$; choose one such neighbor
$v_\ell\in B_i^\ell\cap I_2.$
Then
$e\cup\{v_\ell\}\in E(H_E^\ell).$
By the fact, no two vertices of $e$ belong to the same $S_j$. Hence~\eqref{eq:strict-between-blocks} gives
$h^{\ell-1}(v_1)> \cdots> h^{\ell-1}(v_{\ell-1}).$
For each $1\le r\le\ell-2$, the strict inequality above and~\eqref{eq:new-h} imply that the maximum of $\delta_z$ on $[v_r,j^*)$ is attained before $v_{r+1}$, and therefore
$\delta(x_{v_r},x_{v_{r+1}}) =h^{\ell-1}(v_r).$
Since $v_\ell\in I_2$ and every $\delta_z$ with $z\in[j^*,L_{t_2})$ is strictly smaller than $\delta_{j^*-1}$, we obtain
$\delta(x_{v_{\ell-1}},x_{v_\ell}) =h^{\ell-1}(v_{\ell-1}).$
Thus the induced $\delta$-sequence of
$x_{v_1},\ldots,x_{v_{\ell-1}},x_{v_\ell}$
is exactly the strictly decreasing sequence
$h^{\ell-1}(v_1),\ldots,h^{\ell-1}(v_{\ell-1})$. Since $h^\ell$ is a red monochromatic embedding of $H_E^\ell$,
\[
        \phi_m^{(\ell)}
        (\{x_{v_1},\ldots,x_{v_{\ell-1}},x_{v_\ell}\})=\red.
\]
The monotone branch of the stepping-up rule now yields
\[
\begin{aligned}
\red
&=\phi_m^{(\ell)}
        (\{x_{v_1},\ldots,x_{v_{\ell-1}},x_{v_\ell}\})\\
&=\phi_m^{(\ell-1)}
        (\{h^{\ell-1}(v_1),\ldots,h^{\ell-1}(v_{\ell-1})\}).
\end{aligned}
\]
Thus $h^{\ell-1}$ is a red monochromatic embedding of $H_E^{\ell-1}$. This verifies condition~\textup{(iii)} and completes the induction.
\end{proof}

The reverse induction has now reached the terminal $3$-uniform level, where the robust random lemma gives the final contradiction. We apply Claim~\ref{claim:induction} with $\ell=3$. There are pairwise disjoint sets $U^3,W_4,\ldots,W_k$ with $|U^3|=u_3$ and $|W_i|=u_{i-1}$, together with a map
$h^3:U^3\to[0,M_3(m))$
which is an almost red embedding of
$H_R^3=H_R(U^3;W_4,\ldots,W_k)$
into $\phi_m^{(3)}[b_3]$.

All these sets have size at least $\alpha_3n\ge\varepsilon n$. Moreover,
\[
        b_3=(10^3)^{-k+3}\frac nm
        =\frac{\alpha_3n}{m}
        \le\frac{u_3}{m}.
\]
Thus the fiber bound in Claim~\ref{claim:induction} gives
\[
        |(h^3)^{-1}(y)|\le b_3\le\frac{|U^3|}{m}
        \qquad\text{for every }y.
\]
Lemma~\ref{lem:robust-random} therefore gives an edge whose three $U^3$-images form a blue triple. These images are necessarily pairwise distinct, since every triple with a repeated image is red at the base level, contradicting the fact that $h^3$ is an almost red embedding. Therefore no monochromatic embedding of $H$ into $\phi_m^{(k)}[b_k]$ exists, and
\[
        r(H)>M_k(m)\lfloor b_k\rfloor.
\]
The comparison~\eqref{eq:final-tower-comparison} now yields
\[
        r(H)>M_k(m)\lfloor b_k\rfloor
        >\tw_{k-1}\!\bigl(c_k\Delta\log\log\Delta\bigr)n.
\]
Finally, by decreasing $c_k$ if necessary, we may ensure that $1/c_k$ exceeds every threshold on $\Delta$ used above; this ensures that the hypothesis $\Delta\ge1/c_k$ implies every lower bound on $\Delta$ used in the proof. This proves the theorem.
\end{proof}

\section{Concluding remark}\label{sec:conclusion}

Theorem~\ref{main} gives
\[
        r(H)\ge\tw_{k-1}\!\bigl(c_k\Delta\log\log\Delta\bigr)n
\]
for bounded-degree $k$-graphs. The proof separates three roles that were previously intertwined: $H_R$ supplies a robust red obstruction at the ordered $3$-uniform base level, $H_E$ provides the structural input for the reverse induction, and the small complete hypergraph $H_B$ excludes the blue branch directly at uniformity $k$. At the level of tower height, the remaining gap to Problem~\ref{problem:CFS} is one full level. It would be particularly interesting to determine whether the robust ordered-base construction can be strengthened enough to support a base scale of order $\tw_3(c\Delta)$, which would match the tower height predicted by Conlon, Fox and Sudakov.

\section*{Declarations}

\noindent\textbf{Funding} The research of Qizhong Lin is partially supported by the National Key R\&D Program of China (Grant No.~2023YFA1010202) and NSFC (No.~12571361).

\noindent\textbf{Ethics declaration} Not applicable.

\noindent\textbf{Author contributions} Chunchao Fan performed the formal analysis, carried out the proofs, and wrote the initial draft of the manuscript. Qizhong Lin supervised the research, provided guidance on the overall strategy, and contributed to the refinement of the technical arguments. Both authors reviewed, revised, and approved the final manuscript.

\noindent\textbf{Data availability} Not applicable.

\noindent\textbf{Competing interests} The authors declare that they have no competing interests.

\section*{Declaration on the use of AI}
The authors utilized ChatGPT in proving Theorem \ref{thm:base-three} and in enhancing the exposition of this paper. All mathematical arguments, results, and conclusions have been reviewed and verified by the authors, who assume full responsibility for the content of this work.

\end{document}